\documentclass[12pt,a4paper]{article}

\usepackage[left=2.5cm, top=2.5cm,bottom=2.5cm,right=2.5cm]{geometry}

\usepackage{amsmath,amssymb,amsthm,mathrsfs,calc,graphicx,stmaryrd,xcolor,dsfont}

\usepackage[british]{babel}
\usepackage{amsfonts}

\newtheoremstyle{definition}
{10pt}
{10pt}
{}
{}
{\bfseries}
{}
{.5em}
{}

\newtheoremstyle{plain}
{10pt}
{10pt}
{\itshape}
{}
{\bfseries}
{}
{.5em}
{}

\theoremstyle{plain}	
\newtheorem{thm}{Theorem} 
\newtheorem{lem}[thm]{Lemma}
\newtheorem{prop}[thm]{Proposition}

\theoremstyle{definition}	

\newtheorem{remark}[thm]{Remark}

\allowdisplaybreaks

\newcommand{\RR}{\mathbb{R}}      

\newcommand{\vol}{\operatorname{vol}}

\def\BB{\mathbb{B}}

\def\EE{\mathbb{E}}

\def\NN{\mathbb{N}}
\def\PP{\mathbb{P}}

\def\RR{\mathbb{R}}

\def\BBn{\mathbb{B}^n}
\def\SSn{\mathbb{S}^{n-1}}

\def\cA{\mathcal{A}}

\def\cF{\mathcal{F}}

\def\cH{\mathcal{H}}
\def\cHn{\mathcal{H}^{n-1}}

\def\dint{\textup{d}}

\def\as{{\rm as}}

\begin{document}

\title{\bfseries Surface area deviation between\\ smooth convex bodies and polytopes}

\author{Julian Grote\footnotemark[1], Christoph Th\"ale\footnotemark[2]\, and Elisabeth M. Werner\footnotemark[3]}

\date{}
\renewcommand{\thefootnote}{\fnsymbol{footnote}}
\footnotetext[1]{University of Ulm, Institute of Stochastics, Ulm, Germany.\\ E-mail: julian.grote@uni-ulm.de}

\footnotetext[2]{Ruhr University Bochum, Faculty of Mathematics, Bochum, Germany.\\ E-mail: christoph.thaele@rub.de}

\footnotetext[3]{Department of Mathematics, Case Western Reserve University, Cleveland, USA.\\ E-mail: elisabeth.werner@case.edu}

\maketitle

\begin{abstract}
The deviation of a general convex body with twice differentiable boundary and an arbitrarily positioned polytope with a given number of vertices is studied. The paper considers the case where the deviation is measured in terms of the surface areas of the involved sets, more precisely, by what is called the surface area deviation. The proof uses arguments and constructions from probability, convex and integral geometry. The bound is closely related to $p$-affine surface areas.
\\[0.5em]
\end{abstract}


\section{Introduction and results}

\subsection{Introduction}

Approximation of convex bodies by polytopes belongs to the most classical and fundamental topics studied in convex and discrete geometry, and is an active area of current mathematical research. The reason for this can be explained by the fact that results in this direction are directly relevant for estimating the complexity of geometric algorithms, see \cite{Edelsbrunner,GardnerKiderlen}, for example. Since best approximating polytopes are rarely accessible directly one usually resorts to random constructions. In fact, considering the volume, the surface area, or more generally, the intrinsic volumes, random polytopes show on average the same behaviour as best approximating polytopes, see \cite{GruberBestApprox1,GruberBestApprox2,GruberHandbook,ReitznerBestApproximation}. This philosophy has been taken up by many authors, who typically impose further restrictions on the position of the polytopes relative to the given convex body. Most classically, the approximating random polytope is constructed as the convex hull of (a large number of) independent random points, which are uniformly distributed in the interior of the convex body, see the survey articles \cite{BaranySurvey,HugSurvey,ReitznerSurvey} and the references cited therein. Also studied is the case where the random points are selected on the boundary of the convex body according to the normalized surface measure, see \cite{Reitzner,RichardsonVuWu,SchuettWerner2003}. The obvious advantage of the latter model is that, when the body has strictly positive Gauss curvature almost everywhere,  with probability one each of the random points is automatically a vertex of the random polytope.

On the other hand, arbitrarily positioned polytopes, i.e.\ polytopes with no restrictions on their location in space relative to the given convex body, are only rarely studied in the literature. 
A lower bound in the symmetric difference metric was given in \cite{Boeroetzky2000} for sufficiently smooth convex bodies. An upper bound in the symmetric difference metric was established in \cite{LudwigSchuettWerner} via a random construction, again in the sufficiently smooth case. There is however still a gap by a factor of dimension between upper and lower bound. In \cite{GroteWerner} a random approach with  an arbitrary density function was proposed. This  allows to discuss extremal problems and also relationships to $p$-affine surface areas. However, the symmetric volume difference is not the only measurement to determine the closeness between a convex body and an approximating polytope. It is equally natural to consider a quantity related to the surface areas of the involved sets, see \cite{BoroczkyCsikos,BoroczkyReitzner,Gromer}. As in the case of the symmetric volume difference, also particularly positioned polytopes were studied initially in the literature. The first random construction for arbitrarily positioned polytopes was carried out in \cite{HoehnerSchuettWerner}, but  only in the case where the underlying convex body is the $n$-dimensional Euclidean unit ball. The principal goal of this paper is to generalize the results from \cite{HoehnerSchuettWerner} to arbitrary convex bodies with sufficiently smooth boundaries. Our construction depends on a special density function, which is intimately related to $p$-affine surface areas. On the technical side dealing with the surface area instead of the symmetric volume difference causes several new difficulties which we overcome in this text. We formally describe the set-up and present our main result in the next subsection.

\subsection{Result}

Let $K$ be a convex body in $\RR^n$, $n\ge 2$, that is of class $C^2_+$. More explicitly, this means that $K$ is a compact convex subset of $\RR^n$ with non-empty interior and such that its boundary $\partial K$ is a twice differentiable $(n-1)$-dimensional sub-manifold of $\RR^n$ whose {Gaussian curvature} $\kappa(x)$ is strictly positive for all $x\in\partial K$. We recall that $\kappa(x)$ is the product of the principal curvatures at $x\in\partial K$, while the {mean curvature} $H(x)$ is $1/n$ times the sum of the principal curvatures at $x$. Moreover, the support function of $K$ is denoted by $h_K$, that is, $h_K(u)=\max\{\langle x,u\rangle:x\in K\}$ for $u\in\SSn$. Let us denote in this paper by $\cHn$ the $(n-1)$-dimensional Hausdorff measure. Finally, the {surface area deviation} between two convex bodies $K,L\subset\RR^n$ is defined as 
\begin{equation} \label{sdm}
\Delta_s(K , L):=\cHn\left( \partial (K \cup L) \right)-\cHn \left(\partial(K \cap L) \right). 
\end{equation}
For the presentation of our main result we also need to recall from \cite{HugAffineSurfaceArea,Lutwak1991,MeyerWerner,SchuettWernerPAffineSurfaceAreas,Werner2007}, for example, that for $p\in[-\infty,\infty]\setminus\{-n\}$ the $p$-affine surface area of a convex body $K$, having the origin as an interior point, is given by
\begin{equation}\label{eq:Defas_p}
\as_p(K) := \int_{\partial K}{\kappa(x)^{p\over n+p}\over\langle x,N(x)\rangle^{n(p-1)\over n+p}}\,\cHn(\dint x),
\end{equation}
where $N(x)$ denotes the unique outer unit normal vector at $x$ and where $\langle\,\cdot\,,\,\cdot\,\rangle$ stands for the standard scalar product in $\RR^n$. Let us mention that $p$-affine surface areas are intensively studied quantities in convex geometry and convex geometric analysis. In particular, they play a central role in what is called $L^p$-Brunn-Minkowski theory, see \cite{HaberlSchuster,LYZ2000,WernerYe}, for example.

Our main result provides an upper bound on the surface area deviation between a convex body $K\subset\RR^n$ of class $C_+^2$ and a polytope with a prescribed number of (sufficiently many) vertices. We emphasize that no restriction on the position of the polytope with respect to the given body is required. Especially, we do not assume that the polytope is contained in or contains $K$.

\begin{thm}\label{mainresult}
	Let $K$ be a convex body in $\RR^n$, $n \ge 2$, that is of class $C^2_+$.  Then, there exists a number $N_{K}\in\NN$ depending only on $K$ such that for all $N\geq N_{K}$ there exists a polytope $P$ in $\RR^n$ with exactly $N$ vertices such that
		\begin{equation}\label{maincor}
		\Delta_s(K, P) \le a\, n\, N^{-\frac{2}{n-1}}\, \as_n(K)^\frac{2}{n-1} \  \cHn(\partial K),
		\end{equation}
		where $a \in (0,\infty)$ is an absolute constant. 
\end{thm}

Theorem \ref{mainresult}  should be compared to some related results known from the literature.
To do so, note first that  the $p$-affine isoperimetric inequality \cite{Luwak1996, WernerYe} for $p \geq 0$ states that 
\begin{equation} \label{aip-Gl}
\frac{\as_p(K)}{\as_p(\BBn)} \leq \left(\frac{\vol(K)}{\vol(\BBn)}\right)^\frac{n-p}{n+p}, 
\end{equation}
with equality if and only if $K$ is an ellipsoid.  Here, and throughout the paper, $\BBn$ denotes the Euclidean unit ball.
In particular, for $p=n$ this means that
$$
\as_n(K) \leq \as_n(\BBn).
$$
We put this into \eqref{maincor} and note also that $\as_n(\BBn) =  \cHn (\partial \BBn)$. Using in addition that, by Stirling's formula, the quantity $\cHn(\partial\BBn)^{2\over n-1}$ is bounded by a constant multiple of $1/n$, this yields the existence of a polytope $P$ in $\RR^n$ with exactly $N$ vertices (with $N$ being sufficiently large) such that
\begin{equation}\label{maincor1}
	\Delta_s(K, P) \le a\,  N^{-\frac{2}{n-1}}\, \ \cHn(\partial K).
	\end{equation}
As mentioned in the previous section, the surface area deviation of the $n$-dimensional Euclidean unit ball $\BBn$  and an arbitrarily positioned polytope was treated in \cite{HoehnerSchuettWerner}. There it was shown that for sufficiently large $N$ one can find a polytope $P$ in $\RR^n$ with precisely $N$ vertices and such that
\begin{equation}\label{eq:HoehnerSchuettWerner}
\Delta_s(\BBn,P) 
\leq a\,N^{-{2\over n-1}}\,\cHn(\partial\BBn)
\end{equation}
with an absolute constant $a>0$.  Thus inequality \eqref{maincor1} is the exact analogue for general convex bodies to the case of the Euclidean ball. In fact,  the even slightly stronger inequality \eqref{maincor} holds.

\par
\noindent
In contrast to surface area deviations, volume deviations have been studied more intensively in the literature. For two convex bodies $K,L\subset\RR^n$ we define
$$
\Delta_v(K,L):=\vol(K\cup L)-\vol(K\cap L).
$$
In \cite{GroteWerner} the authors derived an upper bound for the volume deviation of a convex body of class $C_+^2$ and an arbitrarily positioned polytope with a prescribed number of vertices. More precisely, they show that if $K\subset\RR^n$ is a convex body of class $C_+^2$ and if $f:\partial K\to\RR_+$ is a continuous and strictly positive function with $\int_{\partial K} f(x)\cHn(\dint x)=1$, there exists a constant $N_{K,f}\in\NN$ depending only on $K$ and on $f$ such that for all $N\geq N_{K,f}$ one can find a polytope $P_f\subset\RR^n$ with precisely $N$ vertices and such that
\begin{align*}
\Delta_v(K,P_f) \leq a\,N^{-{2\over n-1}}\int_{\partial K}{\kappa(x)^{1\over n-1}\over f(x)^{2\over n-1}}\,\cHn(\dint x)
\end{align*}
for some absolute constant $a\in(0,\infty)$. 
Especially, taking for $f$ the function
$$
f(x) = {\kappa(x)^{1\over n+1}\over \int_{\partial K}\kappa(x)^{1\over n+1}\,\cHn(\dint x)}={\kappa(x)^{1\over n+1}\over \as_1(K)},\qquad x\in\partial K,
$$
yields the existence of an $n$-dimensional polytope $P$ with precisely $N$ vertices (again $N$ sufficiently large) such that
\begin{align}\label{eq:VolumeDeviation}
\Delta_v(K,P) \leq a\, N^{-{2\over n-1}}\,\as_1(K)^{n+1\over n-1}.
\end{align}
A comparison of Theorem \ref{mainresult} and \eqref{eq:VolumeDeviation} shows that the difference between the upper bound for the volume and surface area deviation consists in the appearance of an additional dimension factor $n$ as well as in the replacement of the (classical) affine surface area $\as_1(K)$ by $\as_n(K)$, which is raised to a different power.

\begin{remark}\label{rem:IntrinsicVolumes}
The surface area $\cHn(\partial K)$ of a convex body $K\subset\RR^n$ can be identified with $2$ times the intrinsic volume $V_{n-1}(K)$ of order $n-1$ of $K$ (see \cite{SchneiderBook} for an introduction to intrinsic volumes). Similarly to the volume and the surface area deviation one can also define for each $i\in\{1,\ldots,n\}$ the $i$th intrinsic volume deviation between two convex bodies $K,L\subset\RR^n$ as
$$
\Delta_i(K,L) := V_i(K\cup L)-V_i(K\cap L).
$$
Using methods that are similar as those for the proof of Theorem \ref{mainresult} presented below one can generalize \eqref{maincor} to an upper bound for $\Delta_i(K,P)$, where $P$ is a polytope with sufficiently many vertices. However, in the present text we concentrate on the surface deviation and treat general intrinsic volumes in a future work.
\end{remark}

The remaining parts of the paper are structured as follows. In Section \ref{sec:AuxResults} we present some auxiliary result. In particular, we obtain there a precise asymptotic formula for the expected surface area of a random polytope, which will turn out to be crucial for our approach. We also rephrase there a Blaschke-Petkantschin-type formula for the Hausdorff measure. The final Section \ref{sec:Proof} is devoted to the proof of Theorem \ref{mainresult}.

\section{Auxiliary results}\label{sec:AuxResults}

\subsection{Precise asymptotics for the expected surface area of random polytopes}\label{subsec:PreciseExpectation}

As already explained above, our proof of Theorem \ref{mainresult} is based on the probabilistic method and hence on the construction of a random polytope (which we implicitly assume to be defined on a suitable probability space $(\Omega,\cA,\PP)$). In particular, it will be important for us to have a precise control over the expected surface area of the convex hull of a fixed (but large) number of random points on the boundary of a convex body $K$, which are chosen independently and according to a continuous and everywhere positive density on $\partial K$. To present the statement, let us write $a_N\sim b_N$ for two sequences $(a_N)_{N\in\NN}$ and $(b_N)_{N\in\NN}$ provided that $a_N/b_N\to 1$, as $N\to\infty$. Also, we shall write $\EE$ for the expectation with respect to the underlying probability measure $\PP$.

\begin{prop}\label{c_{i,n}}
Let $K\subset\RR^n$ be a convex body of class $C^2_+$. Let $P_N$ be the convex hull of $N$ points chosen independently at random according to a continuous and positive probability density $f:\partial K\to\RR_+$. Then
\begin{align*}
\cHn(\partial K) - \EE[\cHn(\partial P_N)] \sim & {N^{-{2\over n-1}}\over \pi(n+1)(n-2)!}{\Gamma(n+{2\over n-1})\over\Gamma({n-1\over 2})}\Gamma\Big({n+1\over 2}\Big)^{n+1\over n-1}\\
&\qquad\qquad\times\int\limits_{\partial K}{\kappa(x)^{1\over n-1}\over f(x)^{2\over n-1}}\,H(x)\,\cHn(\dint x).
\end{align*}
\end{prop}
\begin{proof}
The proof is based on a combination of two known results. First, Theorem 1 in \cite{Mueller} (see also \cite{Affentranger}) states that if $K=\BBn$ is the $n$-dimensional Euclidean unit ball and $f=\cHn(\partial\BBn)^{-1}$ is the density of the uniform distribution on $\partial\BBn$ then
\begin{equation}\label{eq:Muller}
\cHn(\partial\BBn)-\EE[\cHn(\partial P_N)] \sim {2^{n+1\over n-1}\pi^{{n\over 2}+{1\over n-1}}\over (n+1)(n-2)!}{\Gamma(n+{2\over n-1})\over\Gamma({n-1\over 2})}\Bigg({\Gamma({n+1\over 2})\over\Gamma({n\over 2})}\Bigg)^{n+1\over n-1}\,N^{-{2\over n-1}}.
\end{equation}
On the other hand a special case of Theorem 1 in \cite{Reitzner} states that if $K$ and $f$ are as in the statement of the theorem then there exists a constant $c_n\in(0,\infty)$ only depending on the dimension $n$ such that
\begin{equation}\label{eq:Reitzner}
\cHn(\partial K) - \EE[\cHn(\partial P_N)] \sim c_n\,\int\limits_{\partial K}{\kappa(x)^{1\over n-1}\over f(x)^{2\over n-1}}\,H(x)\,\cHn(\dint x)\,N^{-{2\over n-1}}.
\end{equation}
Especially, taking $K=\BBn$ and $f=\cHn(\partial\BBn)^{-1}$ in \eqref{eq:Reitzner} gives
\begin{equation}\label{eq:ReitznerSpecialcase}
\cHn(\partial\BBn)-\EE[\cHn(\partial P_N)] \sim c_n\,\cHn(\partial\BBn)^{{n+1\over n-1}}\,N^{-{2\over n-1}}.
\end{equation}
Comparing now \eqref{eq:Muller} with \eqref{eq:ReitznerSpecialcase} implies that the constant $c_n$ in \eqref{eq:Reitzner} is given by
\begin{align*}
c_n &= {2^{n+1\over n-1}\pi^{{n\over 2}+{1\over n-1}}\over (n+1)(n-2)!}{\Gamma(n+{2\over n-1})\over\Gamma({n-1\over 2})}\Bigg({\Gamma({n+1\over 2})\over\Gamma({n\over 2})}\Bigg)^{n+1\over n-1}\,\cHn(\partial\BBn)^{-{n+1\over n-1}}\\
&={1\over \pi(n+1)(n-2)!}{\Gamma(n+{2\over n-1})\over\Gamma({n-1\over 2})}\Gamma\Big({n+1\over 2}\Big)^{n+1\over n-1}
\end{align*}
after simplifications. Plugging this into \eqref{eq:Reitzner} proves the claim.
\end{proof}

\begin{remark}
It is clear that Proposition \ref{c_{i,n}} can be generalized to intrinsic volumes of arbitrary order, since the result in \cite{Reitzner} holds in this framework and the precise asymptotics for the intrinsic volumes $n$-dimensional unit ball has been computed in \cite{Affentranger} (in fact, the so-called Quermassintegrals were considered in \cite{Affentranger}, but they are directly linked with the intrinsic volumes up to a multiplicative factor).
\end{remark}

\subsection{Tools from integral geometry}\label{subsec:BlaschkePetkantschin}

Tools from integral geometry will play a crucial role at several places in the proof of Theorem \ref{mainresult}. For completeness we gather them in the present section as not all of them might be well known and since we have in mind a broad readership with different mathematical backgrounds. We start with the following change-of-variables formula taken from \cite[Equation (2.62)]{SchneiderBook}.

\begin{prop}\label{prop:ChangeOfVariables}
Let $K\subset\RR^n$ be a convex body of class $C_+^2$ and let $g:\partial K\to\RR$ be a continuous function. For $u\in\SSn$ we denote by $x(u)\in\partial K$ the unique point with outer unit normal vector $u$, i.e., $u=N(x)$. Then
\begin{align}\label{nvol}
\int\limits_{\SSn}g(x(u))\,\cHn(\dint u) = \int\limits_{\partial K} g(x)\,\kappa(x)\,\cHn(\dint x).
\end{align}
\end{prop}

Next we rephrase a special case of Minkowski's integral formula, see \cite[Equation (5.60)]{SchneiderBook} with $j=1$ there and Remark \ref{rem:Minkowski} below. We recall that $h_K$ stands for the support function of $K$ and $H(x)$ is the mean curvature at $x\in\partial K$.

\begin{prop}
Let $K\subset\RR^n$ be a convex body of class $C_+^2$. Then 
\begin{align}\label{eq:MinkowskiOriginal}
\cHn(\partial K) = \int\limits_{\partial K}h_K(N(x))\,H(x)\,\cHn(\dint x).
\end{align}
\end{prop}

\begin{remark}\label{rem:Minkowski}
The general Minkowski integral formula says that for a convex body $K\subset\RR^n$ of class $C_+^2$ and for $j\in\{1,\ldots,n-1\}$,
$$
\int_{\partial K}H_{j-1}(x)\,\cHn(\dint x) = \int_{\partial K}h_K(N(x))\,H_j(x)\,\cHn(\dint x),
$$
where $H_j(x)$ is the $j$th elementary symmetric function of the principal curvatures of $K$ at $x\in\partial K$. Taking $j=1$ we obtain \eqref{eq:MinkowskiOriginal}.
\end{remark}

Integral-geometric transformation formulas, also known as Blaschke-Petkantschin formulas, are widely used in the theory or random polytopes. While the original Blaschke-Petkantschin formula for the Lebesgue measure can be applied if the random points are distributed in the interior of a given convex body, in our case we need a version for the Hausdorff measure. In a very general form, such a transformation formula was obtained by Z\"ahle \cite{Zähle} using methods from geometric measure theory (see also \cite{VedelJensen} for a more elementary approach under slightly stronger assumptions). The following special case can also directly be derived from the classical Blaschke-Petkantschin formula for the Lebesgue measure by a limiting procedure as demonstrated in \cite{Reitzner}. For points $x_1, \ldots, x_i \in \RR^n$, $i \in \NN$, we denote by $[x_1,\ldots,x_i]$ the convex hull of $x_1,\ldots,x_i$. Especially if $i=n$, $\cHn([x_1,\ldots,x_n])$ is the $(n-1)$-Hausdorff measure of the $(n-1)$-dimensional simplex spanned by $x_1,\ldots,x_n$.

\begin{prop}\label{lem:Zähle}
Let $K\subset\RR^n$ be a convex body of class $C_+^2$ and let $g:\partial K\to\RR_+$ be a continuous function. Then, 
	\begin{align*}
		&\int\limits_{\partial K} \cdots \int\limits_{\partial K} g(x_1,\ldots,x_n)\,\cHn(\dint x_1)\ldots\cHn(\dint x_n)\\
		&  = (n-1)! \int\limits_{\SSn} \int\limits_0^\infty \int\limits_{\partial K \cap H} \cdots \int\limits_{\partial K \cap H} g(x_1,\ldots,x_n)\,\cHn([x_1,\ldots,x_n])\\
		&\qquad \qquad \qquad\qquad\qquad\qquad\times \prod_{j=1}^{n} l_{H}(x_j)\,\cH^{n-2}(\dint x_1)\ldots\cH^{n-2}(\dint x_n)\dint h\cHn(\dint u) ,
	\end{align*}
	where $l_{H}(x_j) := \left\| {\rm proj}_{H} N(x_j)\right\|^{-1}$ denotes the inverse length of the orthogonal projection onto $H$ of the unique outer unit normal vector $N(x_j)$ of $\partial K$ at $x_j$.
\end{prop}

\section{Proof of Theorem \ref{mainresult}}\label{sec:Proof}

\subsection{Preliminaries}

We start by introducing the set-up and some more notation. Throughout this section $K$ will denote a convex body in $\RR^n$ of class $C_2^+$ and $f:\partial K\to\RR_+$ will denote a strictly positive and continuous function, which satisfies $\int_{\partial K}f(x)\,\cHn(\dint x)=1$. In the course of the proof we will specialize $f$ further by taking $f=f_n$ with $f_n$ given by
\begin{equation}\label{eq:Deff_n}
		f_n(x) = \frac{1}{\as_n(K)} \   \frac{ \kappa(x)^\frac{1}{2}}{h_K(u(x))^\frac{n-1}{2}},\qquad x\in\partial K,
\end{equation}
which integrates to $1$ by definition \eqref{eq:Defas_p} of $\as_n(K)$ and since $\langle x,N(x)\rangle=h(u(x))$. We denote by $\PP_f$ the probability measure on $\partial K$ with density $f$ with respect to $\cHn$. That is,
$$
{\dint \PP_f\over\dint\cHn}(x) = {f(x)},\qquad x\in\partial K.
$$
Also, if $H\subset\RR^n$ is a hyperplane with $H \cap K \neq \emptyset$, we denote by $\PP_{f_{\partial K \cap H}}$ the probability measure  on $\partial K \cap H$ with normalized density $f$ with respect to the $(n-2)$-dimensional Hausdorff measure restricted to $\partial K\cap H$, i.e.,
\begin{align*}
{\dint \PP_{f_{\partial K \cap H}} \over \dint\cH^{n-2}} (x) = \frac{f(x)}{\int\limits_{\partial K \cap H} f(y) \,\cH^{n-2}(\dint y)},\qquad x\in \partial K\cap H.
\end{align*}

\subsection{The probabilistic construction}

As in \cite{GroteWerner,LudwigSchuettWerner} we will obtain the approximating polytope $P$ of $K$ by using the probabilistic method. To be more precise, we consider a convex body that is slightly bigger than the original body $K$, then choose $N$ points at random on the boundary of the bigger body and take the convex hull of these points. We shall prove that such a random polytope satisfies the desired property on average and then argue that also a realization with the same property exists.

Without loss of generality we can and will assume that the origin, denoted by $0$, is in the interior of $K$. More specifically, we assume that $0$ coincides with the centre of gravity of $K$. Since the density $f$ lives on the boundary of $K$, we will choose the random points on $\partial K$ and approximate a slightly smaller body, say $(1-c)K$, where $c := c_{n,N}$ depends on the  dimension $n$ and the number of points $N$, and has to be chosen carefully. We compute the expected surface area deviation
$$
\EE[\Delta_s((1-c)K, P_N)]
$$
between $(1-c)K$ and a random polytope $P_N : =[X_1,\ldots,X_N]$ whose vertices $X_1,\ldots,X_N$  are independent and randomly chosen from the boundary of $K$ according to the probability measure $\PP_f$. In order to do this, we choose $c$ such that the following holds:
\begin{equation}\label{choice:c}
\EE[\cHn(\partial P_N)] = \cHn(\partial (1-c) K) = (1-c)^{n-1} \cHn(\partial K).
\end{equation}
By Theorem \ref{c_{i,n}}, we have that
\begin{align*}
\cHn(\partial K) -  \EE[\cHn(\partial P_N)]   \sim  N^{-{2\over n-1}}\,c_{n,K,f},
\end{align*}
as $N\to\infty$, with
$$
c_{n,K,f}:={1\over \pi(n+1)(n-2)!}{\Gamma(n+{2\over n-1})\over\Gamma({n-1\over 2})}\Gamma\Big({n+1\over 2}\Big)^{n+1\over n-1}\int_{\partial K}{\kappa(x)^{1\over n-1}\over f(x)^{2\over n-1}}\,H(x)\,\cHn(\dint x).
$$
Hence, with the choice \eqref{choice:c} of $c$, this yields
\begin{align*}
\cHn(\partial K) - (1-c)^{n-1} \cHn(\partial K)  \sim  N^{-\frac{2}{n-1}}\,c_{n,K,f},
\end{align*}
as $N \rightarrow \infty$, and leads to 
\begin{align}\label{c}
c &\sim N^{-\frac{2}{n-1}}\,{c_{n,K,f}\over (n-1)\cHn(\partial K)}.
\end{align}
In particular, for sufficiently large $N$ we get the lower bound 
\begin{align}\label{AbschätzungfürC}
c \ge \left(1-\frac{1}{n}\right) N^{-\frac{2}{n-1}}\,{c_{n,K,f}\over (n-1)\cHn(\partial K)}.
\end{align}
\vskip 3mm
\noindent

\subsection{A first upper bound for the expected surface area deviation}

In this paper we denote for fixed $u\in \SSn$ and $h\ge 0$ by $H:= H(u,h)$ the hyperplane orthogonal to $u$ and at distance $h$ from the origin. Further, we let $H^+$ be the half-space bounded by $H$ which contains the origin and put
\begin{align}\label{H+}
	\PP_f(\partial K \cap H^+) := \int\limits_{\partial K \cap H^+} f(x) \,\cHn(\dint x).
\end{align} 
Also recall that the support function of a convex body $K$ is denoted by $h_K:\SSn\to\RR$. In what follows, $a\in (0,\infty)$ will always denote an absolute constant whose value might change from occasion to occasion. 

\begin{lem}\label{Schritt1}
For sufficiently large $N$ and for all sufficiently small $ \epsilon \ge c h_K(u)$ we have that
\begin{align*}
&\EE[\Delta_s((1-c)K, P_N)]\\
& \le a\, \binom{N}{n}\, n! \int\limits_{\SSn} \int\limits_{h_K(u) - \epsilon}^{h_K(u)}  \left(\PP_f(\partial K \cap H^+)\right)^{N-n} \frac{1}{h_K(u)} \max \{0, ((1-c)h_K(u) - h)\}\\
&\qquad  \times \int\limits_{\partial K \cap H} \cdots \int\limits_{\partial K \cap H} (\cHn([x_1,\ldots,x_n]))^2 \prod_{j=1}^{n} l_{H}(x_j)\,\cH^{n-2}(x_1)\ldots\cH^{n-2}(\dint x_n)\dint h\cHn(\dint u).
\end{align*}
\end{lem}
\begin{proof}
With the choice \eqref{choice:c} of the parameter $c$ we obtain 
\begin{align*}
&\EE[\cHn(\partial ((1-c)K) \cap P_N)] + \EE[\cHn(\partial ((1-c)K) \cap P_N^c)]\\
&\qquad = \EE[\cHn(\partial P_N \cap (1-c)K)] + \EE[\cHn(\partial P_N \cap ((1-c)K)^c] 
\end{align*}
and thus
\begin{align*}
&\EE[\Delta_s((1-c)K, P_N)]\\
&\qquad = \EE[\cHn(\partial ((1-c)K) \cap P_N^c)] + \EE[\cHn(\partial P_N \cap ((1-c)K)^c]\\
&\qquad \qquad - \EE[\cHn(\partial P_N \cap (1-c)K] -  \EE[\cHn(\partial ((1-c)K) \cap P_N)]\\
&\qquad = 2 \left(\EE[\cHn(\partial ((1-c)K) \cap P_N^c)]  -  \EE[\cHn(\partial P_N \cap (1-c)K] \right).
\end{align*} 
We refer to \cite{HoehnerSchuettWerner} for a similar computation. Therefore,
\begin{align*}
&\EE[\Delta_s((1-c)K, P_N)]\\
&\quad = 2\int\limits_{\partial K} \cdots \int\limits_{\partial K} \cHn(\partial ((1-c)K) \cap P_N^c) - \cHn(\partial P_N \cap (1-c)K) \\
&\hspace{5cm}  \times\PP_f(\dint x_1) \ldots  \PP_f(\dint x_N)\\
&\quad \le  2\int\limits_{\partial K} \cdots \int\limits_{\partial K}  \cHn(\partial ((1-c)K) \cap P_N^c) \mathds{1}_{\{0\in P_N\}} \, \PP_f(\dint x_1) \ldots \PP_f(\dint x_N)\\
&\quad \qquad +  2\int\limits_{\partial K} \cdots \int\limits_{\partial K}  \cHn(\partial ((1-c)K) \cap P_N^c) \mathds{1}_{\{0\notin P_N\}} \, \PP_f(\dint x_1) \ldots  \PP_f(\dint x_N)\\ 
&\quad \qquad -  2\int\limits_{\partial K} \cdots \int\limits_{\partial K} \cHn(\partial P_N \cap (1-c)K)  \mathds{1}_{\{0\in P_N\}}\, \PP_f(\dint x_1) \ldots  \PP_f(\dint x_N)\\ 
&\quad \le  2\int\limits_{\partial K} \cdots \int\limits_{\partial K}  \cHn(\partial ((1-c)K) \cap P_N^c) \mathds{1}_{\{0\in P_N\}} \, \PP_f(\dint x_1) \ldots  \PP_f(\dint x_N)\\
&\quad \qquad -  2\int\limits_{\partial K} \cdots \int\limits_{\partial K} \cHn(\partial P_N \cap (1-c)K)  \mathds{1}_{\{0\in P_N\}} \, \PP_f(\dint x_1) \ldots  \PP_f(\dint x_N)\\ 
&\quad \qquad +  2\, \cHn(\partial K)\, 	\PP_f^N(\{0 \notin [x_1,\ldots,x_N]\}).
\end{align*}
Since the density function $f$ is strictly positive and since the origin is contained in the interior of $K$, it is a standard argument in random polytope theory which implies that
$$
\PP_f^N(\{0 \notin [x_1,\ldots,x_N]\}) \leq e^{-aN},
$$
see \cite{SchuettWerner2003}. In the course of the proof we will see that the difference between the first and the second summand above is of order $N^{-\frac{2}{n-1}}$ and thus it is enough to consider this difference in what follows and to neglect the third, exponentially small term from now on.

For a polytope $P\subset\RR^n$ we denote by $\cF_{n-1}(P)$ the set of all facets of $P$. Also, for $x_1,\ldots,x_n\in\RR^n$ we let 
\begin{align*}
{\rm cone}(x_1,\ldots,x_n) := \left\{\sum_{i=1}^{n} a_i\, x_i : a_i \ge 0, 1\le i\le n   \right\}
\end{align*}
be the cone generated by $x_1,\ldots,x_n$. For a subset $\{j_1,\ldots,j_n\}\subseteq\{1,\ldots,N\}$ we define two functions $\Phi_{j_1,\ldots,j_n} : (\partial K)^n\rightarrow \RR$ and $\Psi_{j_1,\ldots,j_n} : (\partial K)^n \rightarrow \RR$ as follows. We put
$$
\Phi_{j_1,\ldots,j_n} (x_1,\ldots,x_N) := \cHn(\partial((1-c)K) \cap P_N^c \cap \text{cone}(x_{j_1},\ldots, x_{j_n})) 
$$
if $[x_{j_1},\ldots,x_{j_n}]$ is a facet of $P_N$  (i.e.\ if $[x_{j_1},\ldots,x_{j_n}]\in \mathcal{F}_{n-1}(P_N)$) and $0\in P_N$. In the other case that $[x_{j_1},\ldots,x_{j_n}]\notin\mathcal{F}_{n-1}(P_N)$ or $0\notin P_N$ we define $\Phi_{j_1,\ldots,j_n} (x_1,\ldots,x_N)=0$. Similarly, we let
$$
\Psi_{j_1,\ldots,j_n} (x_1,\ldots,x_N) := \cHn(\partial(1-c)K \cap [x_{j_1},\ldots, x_{j_n}]),
$$
provided that $[x_{j_1},\ldots,x_{j_n}] \in \mathcal{F}_{n-1}(P_N)$ and $0 \in P_N$, and put $\Psi_{j_1,\ldots,j_n} (x_1,\ldots,x_N)$ to be zero otherwise.
Conditioned on the event that $0\in P_N$, with probability one we have that $\RR^n$ can be written as the disjoint union
\begin{align*}
\RR^n = \;\;\;\;\;\;\;\;\;\;\;\;\;\;\cdot \hspace{-12pt}\hspace{-1.25cm}\bigcup\limits_{[x_{j_1},\ldots,x_{j_n}] \in \mathcal{F}_{n-1}(P_N)} \text{cone}(x_{j_1},\ldots, x_{j_n}).
\end{align*}
Moreover,
\begin{align*}
&\PP_f^{N-n}(\{(x_{n+1},\ldots,x_N) :  [x_{1},\ldots,x_{n}] \in \mathcal{F}_{n-1}(P_N) \text{ and } 0 \in P_N \})= \left(\PP_f(\partial K \cap H^+)\right)^{N-n},
\end{align*}
where $H=H(x_1,\ldots,x_n)$ is the hyperplane spanned by the points $x_1,\ldots,x_n$ and we recall the definition of $\PP_f(\partial K \cap H^+)$ given in \eqref{H+}. We also notice that the random polytopes $P_N$ are simplicial with probability one. Therefore, and since the set where $H$ is not well defined has measure zero and all $N$ points are independent and identically distributed, we arrive at 
\begin{align*}
&\int\limits_{\partial K} \cdots \int\limits_{\partial K}  \cHn(\partial ((1-c)K) \cap P_N^c) \mathds{1}_{\{0\in P_N\}} \,\PP_f(\dint x_1) \ldots  \PP_f(\dint x_N)\\
&\quad \qquad -  \int\limits_{\partial K} \cdots \int\limits_{\partial K} \cHn(\partial P_N \cap (1-c)K)  \mathds{1}_{\{0\in P_N\}} \, \PP_f(\dint x_1) \ldots \PP_f(\dint x_N)\\ 
&  =\int\limits_{\partial K} \cdots \int\limits_{\partial K} \sum_{\{j_1,\ldots,j_n\} \subseteq \{1,\ldots, N\}} \!\!\!\!\!\![\Phi_{j_1,\ldots,j_n} (x_1,\ldots,x_N) -  \Psi_{j_1,\ldots,j_n} (x_1,\ldots,x_N)]\,\PP_f(\dint x_1) \ldots  \PP_f(\dint x_n)\\
& = \binom{N}{n} \int\limits_{\partial K} \cdots \int\limits_{\partial K} [\Phi_{1,\ldots,n} (x_1,\ldots,x_N) - \Psi_{1,\ldots,n} (x_1,\ldots,x_N)] \, \PP_f(\dint x_1) \ldots  \PP_f(\dint x_N)\\
& = \binom{N}{n} \int\limits_{\partial K} \cdots \int\limits_{\partial K} \left(\PP_f(\partial K \cap H^+)\right)^{N-n}\big[\cHn(\partial((1-c)K) \cap H^- \cap \text{cone}(x_{1},\ldots, x_{n})) \\
&\qquad\qquad\qquad\qquad\qquad - \cHn((1-c)K \cap H \cap [x_{1},\ldots, x_{n}])\big]\, \PP_f(\dint x_1) \ldots  \PP_f(\dint x_n).
\end{align*}
Next, we apply the Blaschke-Petkantschin-type formula presented in Proposition \ref{lem:Zähle} to the last integral expression. For large enough $N$ this leads to the upper bound
\begin{align*}
&\EE[\Delta_s((1-c)K, P_N)]\\
&\quad \le a\, \binom{N}{n}\, (n-1)! \int\limits_{\SSn} \int\limits_0^\infty \int\limits_{\partial K \cap H} \cdots \int\limits_{\partial K \cap H} \left(\PP_f(\partial K \cap H^+)\right)^{N-n} \,\cHn([x_1,\ldots,x_n])\\
&\quad  \times \big[\cHn(\partial((1-c)K) \cap H^- \cap \text{cone}(x_{1},\ldots, x_{n})) - \cHn((1-c)K \cap H \cap [x_{1},\ldots, x_{n}])\big]\\
&\quad  \times  \prod_{j=1}^{n} l_{H}(x_j) \,\PP_{f_{\partial K \cap H}}(\dint x_1) \cdots  \PP_{f_{\partial K \cap H}} (\dint x_n) \dint h \cHn(\dint u).
\end{align*}
We note that for fixed $u\in\SSn$ the integrand on the right-hand side can be non-zero if and only if $0\leq h\leq h_K(u)$, where we recall that $h_K(u)$ is the support function of $K$ in direction $u$. The same argument as in \cite[Page 9]{LudwigSchuettWerner}, \cite[Page 2255]{Reitzner} or \cite{GroteWerner} show that it is enough to consider $h$ for which $h_K(u)-\varepsilon\leq h\leq h_K(u)$, where $\epsilon > 0$ is sufficiently small. In fact, the integral over the remaining interval $[0,h_K(u)-\varepsilon]$ decays exponentially fast in $N$. In particular, for sufficiently large $N$ we can choose $\epsilon$ such that
\begin{align*} 
c h_K(u) \leq \epsilon \le {h_K(u)\over n},
\end{align*}
where $c$ is as in \eqref{c}. From \cite[Page 8]{HoehnerSchuettWerner} we have the inequality 
\begin{align*}
&\cHn(\partial((1-c)K) \cap H^- \cap \text{cone}(x_{1},\ldots, x_{n}))\\
&\qquad \le \left(\frac{(1-c)h_K(u)}{h}\right)^{n-1} \cHn((1-c)K \cap [x_1,\ldots,x_n]),
\end{align*}
which implies that
\begin{align*}
&\cHn(\partial((1-c)K) \cap H^- \cap \text{cone}(x_{1},\ldots, x_{n})) - \cHn((1-c)K \cap H \cap [x_{1},\ldots, x_{n}])\\
&\qquad \le \max \left\{ \left(\frac{(1-c)h_K(u)}{h}\right)^{n-1} - 1, 0\right\} \,\cHn((1-c)K \cap [x_1,\ldots,x_n]).
\end{align*}
Since  $c$ is of the order $N^{- \frac{2}{n-1}}$ and $\epsilon \le h_K(u)/n$, for sufficiently large $N$, we get
\begin{align*}
\frac{(1-c) h_K(u) -h}{h} \le \frac{(1-c) h_K(u) - h_K(u) + \epsilon}{h_K(u) - \epsilon} \le  \frac{\frac{1}{n} - c}{1-\frac{1}{n}} \le \frac{1}{n-1}.
\end{align*}
Thus,
\begin{align*}
&\left(\frac{(1-c)h_K(u)}{h}\right)^{n-1} - 1\\
&\qquad = \left( \frac{h+(1-c)h_K(u)-h}{h} \right)^{n-1} - 1\\
&\qquad = \left(1 + \frac{(1-c)h_K(u) -h}{h}\right)^{n-1}  - 1\\
&\qquad = (n-1)  \frac{(1-c) h_K(u) -h}{h} + \frac{(n-1)(n-2)}{2} \left(\frac{(1-c) h_K(u) -h}{h}\right)^2  + \ldots  \\
&\hspace{4cm} \ldots + \Big(\frac{(1-c) h_K(u) -h}{h}\Big)^{n-1}\\
&\qquad \le (n-1)  \frac{(1-c) h_K(u) -h}{h} \cdot \sum_{k=0}^{\infty} \frac{n^k}{k!} \left(\frac{(1-c) h_K(u) -h}{h}\right)^k\\
&\qquad \le (n-1) \exp\left(\frac{n}{n-1}\right)\, \frac{(1-c) h_K(u) -h}{h}\\
&\qquad\le a\, n\, \frac{(1-c)h_K(u) - h}{h}.
\end{align*}
As a consequence, for sufficiently large $N$,
\begin{align*}
&\cHn(\partial((1-c)K) \cap H^- \cap \text{cone}(x_{1},\ldots, x_{n})) - \cHn((1-c)K \cap H \cap [x_{1},\ldots, x_{n}])\\
&\qquad \qquad \le a\, n\, \frac{1}{h}\, \cHn([x_1,\ldots,x_n]) \, \max \{0, (1-c)h_K(u) - h\}\\
&\qquad \qquad \le a\, n\, \frac{1}{h_K(u)}\, \cHn([x_1,\ldots,x_n]) \, \max \{0, (1-c)h_K(u) - h\},
\end{align*}
since $\frac{1}{h} \le \frac{1}{(1-1/n) h_K(u)} = \frac{n}{n-1} \frac{1}{h_K(u)} \le 2 \frac{1}{h_K(u)}$.
This proves the lemma.
\end{proof}

\subsection{A bound for the inner integral}

In a next step we consider the inner integral over $\partial K\cap H$ appearing in Lemma \ref{Schritt1}. The following upper bound has been derived in  \cite{GroteWerner}. In what follows we use for $k\in\NN$ the notation 
$$
\omega_{k}:=\cH^{k-1}(\mathbb{S}^{k-1}) = {2\pi^{k/2}\over\Gamma({k\over 2})}\qquad\text{and}\qquad\kappa_k:=\vol_k(\BB^k) = {\pi^{k/2}\over\Gamma({k\over 2}+1)}.
$$

\begin{lem}\label{Schritt2}
Fix $u\in\SSn$ and let $x(u)\in\partial K$ be the point with outer unit normal vector $u$. Let $H$ be a hyperplane orthogonal to $u$ at distance $h$ and put $z := h_K(u) - h$. Then, for all sufficiently small $\delta > 0$, 
	\begin{align*}
	&\int\limits_{\partial K \cap H} \cdots \int\limits_{\partial K \cap H} (\cHn([x_1,\ldots,x_n]))^2 \prod_{j=1}^{n} l_{H}(x_j)\, \PP_{f_{\partial K \cap H}}(\dint x_1) \cdots  \PP_{f_{\partial K \cap H}} (\dint x_n)\\
	&\le (1+\delta)^{\frac{n(n+3)}{2}}\,  2^{\frac{n^2-n-2}{2}}\,  z^{\frac{n^2-n-2}{2}} \frac{n \   \omega_{n-1}^n}{(n-1)!\, (n-1)^{n-1}}\, f(x(u))^n\, \kappa(x(u))^{-\frac{n}{2} - 1}+ \delta O(z^{\frac{n^2-n-2}{2}}),
	\end{align*}
	where the constant in the $O(\,\cdot\,)$-term can be chosen independently of $x(u)$ and $\delta$.
\end{lem}

\subsection{A decomposition into two terms $T_1$ and $T_2$}

Using the upper bound provided in Lemma \ref{Schritt2}, we now decompose the expected surface area deviation between $(1-c)K$ and $P_N$ into two integral terms, which will be treated separately afterwards. In order to do this, we put $s := \PP_f(\partial K \cap H^-)$, or, in other words, $\PP_f(\partial K \cap H^+) = 1-s$. Also recall the definition of $z$ from Lemma \ref{Schritt2}.

\begin{lem}\label{Schritt2a}
For sufficiently large $N$ and sufficiently small $\delta > 0$, we have 
\begin{align*}
\EE[\Delta_s((1-c)K, P_N)] \le T_1+T_2
\end{align*}
with the terms $T_1$ and $T_2$ given by
\begin{align*}
T_1 &:= (1+\delta)^{\frac{3n^2 + 3n}{2}}\, a\, \binom{N}{n}\, n^2\, \int\limits_{\SSn}  \frac{1}{h_K(u) \kappa(x(u))} \\
&\qquad\qquad\qquad\times \int\limits_{0}^{1} (1-s)^{N-n}\, s^{n-1}\, (z-ch_K(u))\, \dint s \cHn(\dint u)
\end{align*}
and 
\begin{align*}
T_2 &:= (1+\delta)^{\frac{3n^2 + 3n}{2}}\, a\, \binom{N}{n}\, n^2\, \int\limits_{\SSn}  \frac{1}{h_K(u) \kappa(x(u))}\\
&\qquad \qquad \times \int\limits_{0}^{s(c h_K(u))} (1-s)^{N-n}\, s^{n-1}\, (ch_K(u) - z)\, \dint s \cHn(\dint u).
\end{align*}
Here, $z=z(s)$ and $s(ch_K(u)) = \int\limits_{\partial K \cap H^-} f(x) \,\cHn(\dint x)$, where $H$ is the unique hyperplane with unit normal vector $u\in \SSn$ and distance $(1-c)h_K(u)$ from the origin, and $H^-$ the half-space bounded by $H$ not containing the origin.  
\end{lem}

We prepare the proof of Lemma \ref{Schritt2a} with the following result taken from \cite{Reitzner}. It provides a bound for $s$ (recall the paragraph before Lemma \ref{Schritt2a}) as well as an upper bound for $z=z(s)$ and its derivative (recall the definition of $z$ from Lemma \ref{Schritt2}).
	
\begin{lem}\label{lem:Reitzner2}
Let $u\in\SSn$ and $x(u)\in\partial K$ be the unique point with outer unit normal vector $u$. Then, for all sufficiently small $\delta > 0$, it holds that 
	\begin{align}\label{s}
	\begin{split}
	&(1+\delta)^{-n}\, 2^{\frac{n-1}{2}}\, f(x(u))\, \kappa(x(u))^{- \frac{1}{2}}\, \kappa_{n-1}\, z^{\frac{n-1}{2}}\\ 
	&\qquad \le s \le (1+\delta)^{n+1}\, 2^{\frac{n-1}{2}}\, f(x(u))\, \kappa(x(u))^{- \frac{1}{2}}\, \kappa_{n-1}\, z^{\frac{n-1}{2}}.
	\end{split}
	\end{align}
	Therefore,
	\begin{align}\label{z}
	z \le (1+\delta)^{\frac{2n}{n-1}}\, \frac{\kappa(x(u))^{1\over n-1}\, (n-1)^{\frac{2}{n-1}}}{2\, f(x(u))^{2\over n-1}\, \omega_{n-1}^{2\over n-1}}\, s^{2 \over n-1}
	\end{align}
	and
	\begin{align}\label{dz}
	\frac{\dint z}{\dint s} \le (1+\delta)^{n}\, \frac{\kappa(x(u))^{1\over 2}\, 2^{-\frac{n-3}{2}}}{f(x(u))\, \omega_{n-1}}\, z^{- \frac{n-3}{2}}.
	\end{align}
\end{lem}

\begin{proof}[Proof of Lemma \ref{Schritt2a}]
Observe first that $\max \{0, (1-c)h_K(u) - h\} = 0$ whenever $h > (1-c)h_K(u)$. This observation together with  Lemma \ref{Schritt1},  Lemma \ref{Schritt2} and the substitution $z = h_K(u) - h$ imply that 
\begin{align*}
&\EE[\Delta_s((1-c)K, P_N)]\\
&\le (1+\delta)^{\frac{n(n+3)}{2}}\, a\, 2^{\frac{n^2-n-2}{2}}\, \binom{N}{n}\, \frac{n^2\, \omega_{n-1}^n}{(n-1)^{n-1}} \int\limits_{\SSn} f(x(u))^n\, \kappa(x(u))^{-\frac{n}{2}-1}\\
&\qquad \times \int\limits_{h_K(u) - \epsilon}^{(1-c)h_K(u)} \left(\PP_f(\partial K \cap H^+)\right)^{N-n} z^{\frac{n^2 - n - 2}{2}}  \frac{1}{h_K(u)} ((1-c)h_K(u) - h) \, \dint h \cHn(\dint u)\\
&\quad + \delta \binom{N}{n} n! \int\limits_{\SSn} \int\limits_{h_K(u) - \epsilon}^{(1-c)h_K(u)} \left(\PP_f(\partial K \cap H^+)\right)^{N-n} O(z^{\frac{n^2 - n - 2}{2}})\\
&\qquad \times  \frac{1}{h_K(u)} ((1-c)h_K(u) - h) \,\dint h \cHn(\dint u)\\
&= (1+\delta)^{\frac{n(n+3)}{2}}\, a\, 2^{\frac{n^2-n-2}{2}}\, \binom{N}{n}\, \frac{n^2\, \omega_{n-1}^n}{(n-1)^{n-1}}\, \int\limits_{\SSn} f(x(u))^n\,\kappa(x(u))^{-\frac{n}{2}-1}\\
&\qquad \times \int\limits_{ch_K(u)}^{\epsilon} \left(\PP_f(\partial K \cap H^+)\right)^{N-n} z^{\frac{n^2 - n - 2}{2}}  \frac{1}{h_K(u)} (z - c h_K(u)) \,\dint z \cHn(\dint u)\\
&\quad+ \delta \binom{N}{n} n! \int\limits_{\SSn} \int\limits_{ch_K(u)}^{\epsilon} \left(\PP_f(\partial K \cap H^+)\right)^{N-n} O(z^{\frac{n^2 - n - 2}{2}}) \frac{1}{h_K(u)} (z -c h_K(u))\, \dint z \cHn(\dint u).
\end{align*}
It will turn out that, as $N\to\infty$, both summands are of order $N^{- \frac{2}{n-1}}$. Since $\delta$ can be chosen arbitrarily small, it is enough to consider the first summand in what follows.

We use  \eqref{dz} and then \eqref{z} to change from $z^{\frac{(n-1)^2}{2}}$ to $s^{n-1}$ and obtain that, for sufficiently large $N$, 
\begin{align*}
&\EE[\Delta_s((1-c)K, P_N)]\\
&\quad \le (1+\delta)^{\frac{n(n+3)}{2} + n}\, a\, 2^{\frac{n^2-n-2}{2}}\, 2^{-\frac{n-3}{2}}\, \binom{N}{n}\, \frac{n^2\, \omega_{n-1}^n}{(n-1)^{n-1}}\, \int\limits_{\SSn} f(x(u))^{n-1}\, \kappa(x(u))^{-\frac{n}{2}-\frac{1}{2}}\\
&\qquad \qquad \times \int\limits_{s(ch_K(u))}^{1} (1-s)^{N-n}\, z^{\frac{n^2 - n - 2 - n + 3}{2}}  \frac{1}{h_K(u)} (z - c h_K(u)) \,\dint s \cHn(\dint u)\\
&\quad \le (1+\delta)^{\frac{n^2 + 5n}{2}}\, a\, 2^{\frac{n^2-2n+1}{2}}\,  \binom{N}{n}\, \frac{n^2\, \omega_{n-1}^{n-1}}{(n-1)^{n-1}}\, \int\limits_{\SSn} f(x(u))^{n-1}\, \kappa(x(u))^{-\frac{n}{2}-\frac{1}{2}}\\
&\qquad \qquad \times \int\limits_{s(ch_K(u))}^{1} (1-s)^{N-n}\, z^{\frac{(n-1)^2}{2}}  \frac{1}{h_K(u)} (z - c h_K(u)) \,\dint s \cHn(\dint u)\\
&\quad \le (1+\delta)^{\frac{n^2 + 5n}{2} + n(n-1)}\, a\, 2^{\frac{(n-1)^2}{2}}\, 2^{-\frac{(n-1)^2}{2}}\,   \binom{N}{n}\, n^2\, \int\limits_{\SSn}  \kappa(x(u))^{-1}\\
&\qquad \qquad \times \int\limits_{s(ch_K(u))}^{1} (1-s)^{N-n}\, s^{n-1}\,  \frac{1}{h_K(u)} (z - c h_K(u)) \,\dint s \cHn(\dint u)\\
&\quad \le  (1+\delta)^{\frac{3n^2 + 3n}{2}}\, a\, \binom{N}{n}\, n^2\, \int\limits_{\SSn}  \kappa(x(u))^{-1}\\
&\qquad \qquad \times \int\limits_{s(ch_K(u))}^{1} (1-s)^{N-n}\, s^{n-1}\,  \frac{1}{h_K(u)} (z - c h_K(u))\, \dint s \cHn(\dint u)\\
&\quad = (1+\delta)^{\frac{3n^2 + 3n}{2}}\, a\, \binom{N}{n}\, n^2\, \int\limits_{\SSn}  \kappa(x(u))^{-1} \int\limits_{0}^{1} (1-s)^{N-n}\, s^{n-1}\\
&\qquad\qquad\times \frac{1}{h_K(u)} (z-ch_K(u))\, \dint s \cHn(\dint u)+ (1+\delta)^{\frac{3n^2 + 3n}{2}}\, a\, \binom{N}{n}\, n^2\, \int\limits_{\SSn} \kappa(x(u))^{-1}\\
&\qquad \qquad \times \int\limits_{0}^{s(c h_K(u))} (1-s)^{N-n}\, s^{n-1}\,  \frac{1}{h_K(u)} (ch_K(u) - z)\, \dint s \cHn(\dint u).
\end{align*} 
In view of the definitions of the terms $T_1$ and $T_2$ this proves the claim.
\end{proof}

\subsection{A bound for the term $T_1$} 

After having decomposed the original integral expression from Lemma \ref{Schritt1} into the sum of $T_1$ and $T_2$, we are now going to bound each of these terms individually. We start with $T_1$ and at the same time start to specialize our set-up by taking the density function $f$ to be equal to $f_n$, which was defined in \eqref{eq:Deff_n}.

\begin{lem}\label{Schritt4}
For sufficiently large $N$ 
we have that
\begin{align*}
T_1 & \le   a\,n\,N^{-{2\over n-1}}\,\as_n(K)^{2\over n-1}\,\cHn(\partial K)
\end{align*}
with an absolute constant $a\in(0,\infty)$.
\end{lem}
\begin{proof}
We apply  \eqref{z} and  \eqref{AbschätzungfürC}, to get that for all sufficiently small $\delta > 0$ and sufficiently large $N$,
\begin{align*}
T_1 & \le (1+\delta)^{\frac{3n^2 + 3n}{2}}\, a\,  \binom{N}{n}\, \frac{n^2}{2}\, \frac{(n-1)^{\frac{2}{n-1}}}{\omega_{n-1}^{2\over n-1}} \\ &\qquad\times\Bigg[
\  (1+\delta)^\frac{2n}{n-1} \int\limits_{\SSn}  \frac{1}{h_K(u)} \frac{\kappa(x(u))^{-1 + \frac{1}{n-1}}}{f(x(u))^{2\over n-1}} \,\cHn(\dint u)\int\limits_{0}^{1} (1-s)^{N-n}\, s^{n-1 + \frac{2}{n-1}}\, \dint s \\
&\qquad\qquad - \left(1-\frac{1}{n}\right) N^{-\frac{2}{n-1}}\, \frac{\Gamma(n+\frac{2}{n-1})}{(n+1) (n-2)!} \frac{1}{\cHn(\partial K)} \int\limits_{\partial K} \frac{\kappa(x)^{\frac{1}{n-1}}}{f(x)^{2 \over n-1}} H(x) \,\cHn(\dint x)\\
&\qquad \qquad\qquad \times \int\limits_{\SSn} \kappa(x(u))^{-1}\,\cHn(\dint x)\int\limits_{0}^{1} (1-s)^{N-n}\, s^{n-1}\, \dint s  \  \Bigg].
\end{align*}
For $u \in \SSn$ let $x=x(u) \in \partial K$ be the point with outer unit normal vector $u$. Using the change-of-variables formula \eqref{nvol} yields that
\begin{align*}
\cHn(\partial K) = \int\limits_{\SSn} \frac{h_K(u) }{\kappa(x(u))} \,\cHn(\dint u)
\end{align*}
and 
$$
\int\limits_{\SSn}  \frac{1}{h_K(u)} \frac{\kappa(x(u))^{-1 + \frac{1}{n-1}}}{f(x(u))^{2\over n-1}}\,\cHn(\dint u) = \int\limits_{\partial K}  \frac{1}{h_K(u(x))} \frac{\kappa(x)^{ \frac{1}{n-1}}}{f(x)^{2\over n-1}}\,\cHn(\dint x).
$$
Together with the observation that
$$
\omega_{n-1}^{2\over n-1} \sim \frac{1}{n}\qquad\text{and}\qquad(n-1)^{2\over n-1} \le 2
$$
we get
\begin{align}
\nonumber &T_1  \le (1+\delta)^{\frac{3n^2 + 3n}{2}}\, a\,  \binom{N}{n}\, \frac{n^3}{2}\\ 
\nonumber &\quad \times \Bigg[ (1+\delta)^\frac{2n}{n-1}  \frac{\Gamma(N-n+1) \Gamma\left(n+ \frac{2}{n-1}\right)}{\Gamma\left(N+1+\frac{2}{n-1}\right)} \int\limits_{\partial K}  \frac{1}{h_K(u(x))} \frac{\kappa(x)^{\frac{1}{n-1}}}{f(x)^{2\over n-1}} \,\cHn(\dint x)\\
\nonumber &\quad\quad - \left(1-\frac{1}{n}\right) N^{-\frac{2}{n-1}}\,\frac{\Gamma(N-n+1) \Gamma\left(n\right)}{\Gamma\left(N+1\right)} \frac{\Gamma(n+\frac{2}{n-1})}{(n+1) (n-2)!} \int\limits_{\partial K} \frac{\kappa(x)^{\frac{1}{n-1}}}{f(x)^{2 \over n-1}} H(x)\,\cHn(\dint x)\Big]\\
\nonumber & \le (1+\delta)^{\frac{3n^2 + 3n}{2}}\, a\,  \binom{N}{n}\, \frac{n^3}{2} \frac{\Gamma(N-n+1) \Gamma\left(n+ \frac{2}{n-1}\right)}{\Gamma\left(N+1+\frac{2}{n-1}\right)}\\ 
\nonumber &\quad \times \Bigg[ (1+\delta)^\frac{2n}{n-1}  \int\limits_{\partial K}  \frac{1}{h_K(u(x))} \frac{\kappa(x)^{\frac{1}{n-1}}}{f(x)^{2\over n-1}} \,\cHn(\dint x)\\
\nonumber &\qquad - \left(1-\frac{1}{n}\right) N^{-\frac{2}{n-1}}\,\frac{\Gamma\left(n\right)}{\Gamma\left(N+1\right)} \frac{\Gamma(N+ 1 + \frac{2}{n-1})}{(n+1) (n-2)!} \int\limits_{\partial K} \frac{\kappa(x)^{\frac{1}{n-1}}}{f(x)^{2 \over n-1}} H(x) \,\cHn(\dint x)\Big]\\
\nonumber & \le (1+\delta)^{\frac{3n^2 + 3n}{2}}\, a\, n^2\, N^{-\frac{2}{n-1}}\Bigg[ (1+\delta)^\frac{2n}{n-1}  \int\limits_{\partial K}  \frac{1}{h_K(u(x))} \frac{\kappa(x)^{\frac{1}{n-1}}}{f(x)^{2\over n-1}} \,\cHn(\dint x)\\
\nonumber &\qquad\qquad\qquad\qquad - \left(1-\frac{1}{n}\right) \frac{\Gamma(n)}{(n+1) (n-2)!} \int\limits_{\partial K} \frac{\kappa(x)^{\frac{1}{n-1}}}{f(x)^{2 \over n-1}} H(x) \,\cHn(\dint x)\Big]\\
\nonumber &\le (1+\delta)^{\frac{3n^2 + 3n}{2}}\, a\, n^2\, N^{-\frac{2}{n-1}}\Bigg[ (1+\delta)^\frac{2n}{n-1}  \int\limits_{\partial K}  \frac{1}{h_K(u(x))} \frac{\kappa(x)^{\frac{1}{n-1}}}{f(x)^{2\over n-1}} \, \cHn(\dint x) \\
\nonumber &\qquad\qquad\qquad\qquad- \left(1-\frac{1}{n}\right) \frac{(n-1)}{(n+1)}  \int\limits_{\partial K} \frac{\kappa(x)^{\frac{1}{n-1}}}{f(x)^{2 \over n-1}} H(x) \,\cHn(\dint x) \Big]\\
\nonumber &\le (1+\delta)^{\frac{3n^2 + 3n}{2}}\, a\, n^2\, N^{-\frac{2}{n-1}}\Bigg[ (1+\delta)^\frac{2n}{n-1}  \int\limits_{\partial K}  \frac{1}{h_K(u(x))} \frac{\kappa(x)^{\frac{1}{n-1}}}{f(x)^{2\over n-1}} \, \cHn(\dint x) \\
\label{eq:BeforeMinkowski}&\qquad\qquad\qquad\qquad- \left(1-\frac{1}{n}\right) \int\limits_{\partial K} \frac{\kappa(x)^{\frac{1}{n-1}}}{f(x)^{2 \over n-1}} H(x) \,\cHn(\dint x)\Bigg],
\end{align}
where in the third last inequality  we have also used that 
\begin{align*}
\frac{\Gamma(N-n+1) \Gamma\left(n+ \frac{2}{n-1}\right)}{\Gamma\left(N+1+\frac{2}{n-1}\right)} \sim \frac{1}{\binom{N}{n} n N^{2\over n-1}}
\end{align*}
and
\begin{align*}
 \Gamma\Big(N+1+\frac{2}{n-1}\Big) \sim N^{2\over n-1} \Gamma\left(N+1\right).
\end{align*}
Now, we replace the generic density $f$ by the particular function $f_n$, which is defined in \eqref{eq:Deff_n}. Together with Minkowski's integral formula \eqref{eq:MinkowskiOriginal} this leads to the bound
\begin{align*}
T_1 & \leq (1+\delta)^{\frac{3n^2 + 3n}{2}}\, a\, n^2\, N^{-\frac{2}{n-1}}\Big[(1+\delta)^{2n\over n-1}\as_n(K)^{2\over n-1}\cHn(\partial K)\\
&\qquad\qquad\qquad-\left(1-{1\over n}\right)\as_n(K)^{2\over n-1}\int_{\partial K}h_K(u(x))\,H(x)\,\cHn(\dint x)\Big]\\
&=(1+\delta)^{\frac{3n^2 + 3n}{2}}\, a\, n^2\, N^{-\frac{2}{n-1}}\,\as_n(K)^{2\over n-1}\,\cHn(\partial K)\Big[(1+\delta)^{2n\over n-1}-\Big(1-{1\over n}\Big)\Big]\\
&\leq a\,n\,N^{-{2\over n-1}}\,\as_n(K)^{2\over n-1}\,\cHn(\partial K),
\end{align*}
where $a$ is some absolute constant. This completes the proof.
\end{proof}

\subsection{A bound for the term $T_2$} 

Now, we deal with the term $T_2$ in Lemma \ref{Schritt2a}. 
\begin{lem}\label{Schritt5}
For sufficiently large $N$, it holds that
\begin{align*}
T_2  \le   a \, \frac{N^{-\frac{2}{n-1}} }{  \sqrt{n}}\,\as_n(K)^{2\over n-1}\,\cHn(\partial K),
\end{align*}
where $a\in (0,\infty)$ is an absolute constant. 
\end{lem}
\begin{proof}
By definition of $T_2$ in Lemma \ref{Schritt2a} and (\ref{s}), 
\begin{align*}
T_2 &\le (1+\delta)^{\frac{3n^2 + 3n}{2}}\, a\, \binom{N}{n}\, n^2\, \int\limits_{\SSn}  
\frac{1}{\kappa(x(u)) \   h_K(u)}\, \int\limits_{0}^{s(c h_K(u))} (1-s)^{N-n}\, s^{n-1}\,  \\
&\qquad \times\left(c h_K(u) - \frac{s^{\frac{2}{n-1} } \   \kappa(x(u)) ^{1\over n-1}}{2 (1+\delta)^{2\frac{n+1}{n-1}}\,  f (x(u)) ^{2\over n-1}  \vol_{n-1}(\BB^{n-1})^{2\over n-1} }  \right) \dint s \cHn(\dint u)\\
&\le (1+\delta)^{\frac{3n^2 + 3n}{2}}\, a\, \binom{N}{n}\, n^2\, \int\limits_{\SSn}  
\frac{1}{\kappa(x(u)) \   h_K(u)}\, \int\limits_{0}^{s(c h_K(u))}  s^{n-1}\,  \\
& \qquad\times\left(c h_K(u) - \frac{s^{\frac{2}{n-1} } \   \kappa(x(u)) ^{1\over n-1} }{2 (1+\delta)^{2\frac{n+1}{n-1}}\,  f (x(u)) ^{2\over n-1}  \vol_{n-1}(\BB^{n-1})^{2\over n-1} }  \right) \dint s \cHn(\dint u)\\
&=  (1+\delta)^{\frac{3n^2 + 3n}{2}}\, a\, \binom{N}{n}\, n^2\,  
\Bigg[ \frac{c}{n} \int\limits_{\partial K}  
 s(c h_K(u))^{n}\,  \cHn(\dint x)    \\
 & -  \frac{(1+\delta)^{-2\frac{n+1}{n-1}}}{ \left(n+\frac{2}{n-1}\right) \   \vol_{n-1}(\BB^{n-1})^{2\over n-1}} \   
  \int\limits_{\partial K}  \frac{\kappa(x) ^{1\over n-1} s(c h_K(u(x)))^{n} \   s(c h_K(u(x)))^{\frac{2}{n-1}}} {2  h_K(u(x)) f (x) ^{2\over n-1} } \cHn(\dint x) \Bigg], 
\end{align*}
where $a\in (0,\infty)$ is an absolute constant and where we have used (\ref{nvol})  in the last equality. 
By Lemma \ref{lem:Reitzner2} we get 
	\begin{align}\label{boundschKu}
			s(c h_K(u(x)))^{\frac{2}{n-1}} &\leq 2 \  (1 + \delta )^{2 \frac{n+1} {n-1}} \  c \  \vol_{n-1}(\BB^{n-1})^{2\over n-1} \  
			\frac{ f (x)^{2\over n-1} \  h_K(u(x)) } {\kappa(x)^{1\over n-1} },
	\end{align}
where $c$ is as in \eqref{c}. Therefore,
\begin{eqnarray}\label{T2}
\nonumber T_2  &\le& a \  (1+\delta)^{\frac{3n^2 + 3n}{2}}\,   c  \  \binom{N}{n}\, n \,  \int\limits_{\partial K}  
 s(c h_K(u))^{n}\,  \cHn(\dint x)  \  \left[ 1-   \frac{(1+\delta)^{-4 \frac{n+1}{n-1}}}{ 1+ \frac{2}{n(n-1)} } \right] \\
&\le& a \  (1+\delta)^{\frac{3n^2 + 3n}{2}}\,   \frac{c}{n(n-1)}   \  \binom{N}{n}\, n \,  \int\limits_{\partial K}   s(c h_K(u))^{n}\,  \cHn(\dint x) .
\end{eqnarray}
By \eqref{c} and Stirling's formula it holds that
	\begin{eqnarray}\label{nochmal:c}
		&& (2\ c)^\frac{n-1}{2}  \  \vol_{n-1}(\BB^{n-1})  \le \frac{1}{N}  \   \vol_{n-1}(\BB^{n-1})  \nonumber \\
		&& \left({2\over \pi(n+1)(n-1)!}{\Gamma(n+{2\over n-1})\over\Gamma({n-1\over 2})}\Gamma\Big({n+1\over 2}\Big)^{n+1\over n-1}\right) ^\frac{n-1}{2} 
		 \left( \frac {\int\limits_{\partial K} \frac{\kappa(x)^{ \frac{1}{n-1}}}{f(x)^{2\over n-1}} H(x) \,\cHn(\dint x)}{\cHn(\partial K)} \right) ^\frac{n-1}{2}  \nonumber \\
	&&= \frac{1}{N} \  \left({2\ n \over (n+1)!}{\Gamma(n+{2\over n-1})\over\Gamma({n-1\over 2})}\Gamma\Big({n+1\over 2}\Big)\right) ^\frac{n-1}{2} 
		 \left( \frac {\int\limits_{\partial K} \frac{\kappa(x)^{ \frac{1}{n-1}}}{f(x)^{2\over n-1}} H(x) \,\cHn(\dint x)}{\cHn(\partial K)} \right) ^\frac{n-1}{2} \nonumber \\
	&& \sim  \frac{n}{e \ N } \ \left( \frac {\int\limits_{\partial K} \frac{\kappa(x)^{ \frac{1}{n-1}}}{f(x)^{2\over n-1}} H(x) \,\cHn(\dint x)}{\cHn(\partial K)} \right) ^\frac{n-1}{2}.
	\end{eqnarray}			
We use this together with  \eqref{boundschKu} in \eqref{T2} and get the bound  		
\begin{eqnarray*}
T_2  &\le& a \  (1+\delta)^{\frac{4n^2 + 4n}{2}}\,    \frac{c}{n(n-1)}  \   \frac{n^n}{e^{n} \ N^n } \  \binom{N}{n}\, n \\
&&\times\left( \frac {\int\limits_{\partial K} \frac{\kappa(x)^{ \frac{1}{n-1}}}{f(x)^{2\over n-1}} H(x) \,\cHn(\dint x)}{\cHn(\partial K)} \right) ^\frac{n(n-1)}{2}	
\int\limits_{\partial K} \frac{f(x)^{n}  \   h_K(u(x))^\frac{n(n-1)}{2}}{\kappa(x)^{ \frac{n}{2}}} \,\cHn(\dint x) \\
&\leq& a \  (1+\delta)^{\frac{4n^2 + 4n}{2}}\,    \frac{c\,n^n}{ (n-1) \, n!\, e^n}  \\
&&\times\left( \frac {\int\limits_{\partial K} \frac{\kappa(x)^{ \frac{1}{n-1}}}{f(x)^{2\over n-1}} H(x) \,\cHn(\dint x)}{\cHn(\partial K)} \right) ^\frac{n(n-1)}{2}	
\int\limits_{\partial K} \frac{f(x)^{n}  \   h_K(u(x))^\frac{n(n-1)}{2}}{\kappa(x)^{ \frac{n}{2}}} \,\cHn(\dint x) \\
&\leq& a \  (1+\delta)^{\frac{4n^2 + 4n}{2}}\,    \frac{c}{ \sqrt{2\pi}\, (n-1) \,\sqrt{n}}  \\
&&\times\left( \frac {\int\limits_{\partial K} \frac{\kappa(x)^{ \frac{1}{n-1}}}{f(x)^{2\over n-1}} H(x) \,\cHn(\dint x)}{\cHn(\partial K)} \right) ^\frac{n(n-1)}{2}	
\int\limits_{\partial K} \frac{f(x)^{n}  \   h_K(u(x))^\frac{n(n-1)}{2}}{\kappa(x)^{ \frac{n}{2}}} \,\cHn(\dint x) , 
\end{eqnarray*}					
		where we  made use of  Stirling's formula once again.
Now we use that $c$ satisfies
$$
c \le a\, \frac{n}{\cHn(\partial K)\, N^{\frac{2}{n-1}} } \  \int\limits_{\partial K} \frac{\kappa(x)^{ \frac{1}{n-1}}}{f(x)^{2\over n-1}} H(x) \,\cHn(\dint x),
$$
together with the elementary inequality ${\sqrt{n}\over\sqrt{2\pi}\,(n-1)}\leq {1\over\sqrt{n}}$ to see that
\begin{eqnarray*}
T_2  &\le&  a \  (1+\delta)^{\frac{4n^2 + 4n}{2}}\,  \frac{1}{  \sqrt{n} \,N^{\frac{2}{n-1}} } \\
&&\times\left( \frac {\int\limits_{\partial K} \frac{\kappa(x)^{ \frac{1}{n-1}}}{f(x)^{2\over n-1}} H(x) \,\cHn(\dint x)}{\cHn(\partial K)} \right) ^\frac{n(n-1)+2}{2}	
\int\limits_{\partial K} \frac{f(x)^{n}  \   h_K(u(x))^\frac{n(n-1)}{2}}{\kappa(x)^{ \frac{n}{2}}} \,\cHn(\dint x) .
\end{eqnarray*}
Now we plug in the special density function $f_n$ given by \eqref{eq:Deff_n} for $f$. Then
$$
\left( \frac {\int\limits_{\partial K} \frac{\kappa(x)^{ \frac{1}{n-1}}}{f(x)^{2\over n-1}} H(x) \,\cHn(\dint x)}{\cHn(\partial K)} \right) ^\frac{n(n-1)+2}{2}	 = \as_n(K)^{{2\over n-1}{n(n-1)+2\over 2}}
$$
and
$$
\int\limits_{\partial K} \frac{f(x)^{n}  \   h_K(u(x))^\frac{n(n-1)}{2}}{\kappa(x)^{ \frac{n}{2}}} \,\cHn(\dint x) = \as_n(K)^{-n}\,\cHn(\partial K),
$$
which yields the bound
$$
T_2 \leq a \, \frac{N^{-\frac{2}{n-1}} }{  \sqrt{n}}\,\as_n(K)^{2\over n-1}\,\cHn(\partial K).
$$
This completes the proof of the lemma.
\end{proof}

\subsection{Completion of the proof of Theorem \ref{mainresult}}

We are now ready to complete the proof of Theorem \ref{mainresult}. Indeed, Lemma \ref{Schritt2a}, Lemma \ref{Schritt4} and Lemma \ref{Schritt5} imply that for sufficiently large $N$,
\begin{equation}\label{last}
\begin{split}
\EE[\Delta_s((1-c)K, P_N)] &\le T_1+T_2\\
& \leq  a_1\,n\,N^{-{2\over n-1}}\,\as_n(K)^{2\over n-1}\,\cHn(\partial K)\\
&\qquad\qquad+a_2 \, \frac{N^{-\frac{2}{n-1}} }{  \sqrt{n}}\,\as_n(K)^{2\over n-1}\,\cHn(\partial K)\\
&\leq a \,n\,N^{-{2\over n-1}}\,\as_n(K)^{2\over n-1}\,\cHn(\partial K)
\end{split}
\end{equation}
where $a,a_1,a_2\in (0,\infty)$ are absolute constants. Taking into account that we were approximating the body $(1-c)K$ instead of $K$, we need to multiply the bound  \eqref{last} by $(1-c)^{-(n-1)}$. Since
\begin{align*}
(1-c)^{n-1} \ge 1 - (n-1)c,
\end{align*}
for sufficiently large $N$ (recall that the definition of $c$ depends on $N$),  we have  that
$$
(1-c)^{-(n-1)} \le a,
$$
where $a\in (0,\infty)$ is another absolute constant. This proves that the \textit{expected} surface area deviation between $K$ and $P_N$ is bounded by the right-hand side in \eqref{last}. In particular, this means that there must exist a realization $P_N(\omega)$ of a polytope with precisely $N$ vertices for some $\omega\in\Omega$ (recall that $(\Omega,\cA,\PP)$ is our underlying probability space) such that the surface area deviation between $P_N(\omega)$ and $K$ is bounded by the same expression. Taking $P$ to be this realization proves the claim of Theorem \ref{mainresult}.\hfill $\Box$

\subsection*{Acknowledgement}

JG was supported by the Deutsche Forschungsgemeinschaft (DFG) via RTG 2131 \textit{High-dimensional Phenomena in Probability -- Fluctuations and Discontinuity}. EW was partially supported by NSF grant DMS-1811146.

\end{document}